\documentclass[12pt]{article}
\usepackage[T2A]{fontenc}
\usepackage[cp1251]{inputenc}
\usepackage[dvips]{graphicx}
\usepackage{amsmath}
\usepackage{amssymb,amsthm}
\newcounter{propcounter}
\newtheorem{theorem}[propcounter]{Theorem}
\newtheorem{lemma}[propcounter]{Lemma}
\newtheorem{proposition}[propcounter]{Proposition}
\newtheorem{corollary}[propcounter]{Corollary}
\newcommand{\ric}{\mathop{\rm Ric}}
\newcommand{\spann}{\mathop{\rm Span}}
\newcommand{\id}{\mathop{\rm Id}}
\newcommand{\trace}{\mathop{\rm Tr}}

\makeatletter

\renewcommand{\section}{\@startsection{section}{1}{\parindent}
{3.5ex plus 1ex minus .2ex}{2.3ex plus .2ex}{\normalsize
\bfseries}}

\makeatother

\begin{document}

\begin{center}
\Large The Gauss Map of Hypersurfaces in $2$-Step Nilpotent Lie
Groups
\end{center}

\vskip 1.5mm

\begin{center}
\large E. V. Petrov
\end{center}

\vskip 1.5mm

\begin{center}
\normalsize \itshape V. N. Karazin Kharkiv National University\\
4 Svobody sq., Kharkiv, 61077, Ukraine\\
E-mail: petrov@univer.kharkov.ua\\
\end{center}

\vskip 1.5mm

\begin{center}
\begin{minipage}[t]{130mm}
\small In this paper we consider smooth oriented hypersurfaces in
$2$-step nilpotent Lie groups with a left invariant metric and
derive an expression for the Laplacian of the Gauss map for such
hypersurfaces in the general case and in some particular cases. In
the case of CMC-hypersurface in the $2m+1$-dimensional Heisenberg
group we also derive necessary and sufficient condi\-ti\-ons for
the Gauss map to be harmonic and prove that for $m=1$ all
CMC-surfaces with the harmonic Gauss map are ''cylinders''.
\end{minipage}
\end{center}

It is proved in \cite{RV} that the Gauss map of a smooth
$n$-dimensional oriented hypersurface in $\mathbb{R}^{n+1}$ is
harmonic if and only if the hypersurface is of a constant mean
curvature (CMC). The same is proved for the cases of $S^3$, which
is a Lie group and thus has a natural definition of the Gauss map
\cite{M1}, and, in different settings, of $H^3$ \cite{M2}. A
generalization of this proposition to the case of Lie groups with
a bi-invariant metric (this class of Lie groups includes, for
example, abelian groups $\mathbb{R}^{n+1}$ and $S^3 \cong SU(2)$)
is proved in \cite{ES}. In this paper we use methods of \cite{ES}
for an investigation of the Gauss map of a hypersurface in some
$2$-step nilpotent Lie group with a left invariant metric. The
theory of such groups is highly developed (see, for example,
\cite{E1} and \cite{E2}).

The paper is organized as follows. After some preliminary
information (section \ref{ch1}), in section \ref{ch2-1} we obtain
an expression for the Laplacian of the Gauss map of a hypersurface
in a $2$-step nilpotent Lie group (Theorem \ref{th1}). Using this
expression we prove some facts concerning relations between
harmonic properties of the Gauss map and the mean curvature of the
hypersurface (see section \ref{ch2-2}), in particular, a
sufficient condition for the stability of CMC-hypersurfaces
(Proposition \ref{prop1-2}). In section \ref{ch3} we consider the
cases of Heisenberg type groups and Heisenberg groups. We show the
harmonicity of the Gauss map of a hypersurface in such groups is,
in general, not equivalent to the constancy of the mean curvature.
Also we obtain necessary and sufficient conditions for this
equivalence in the particular case of Heisenberg groups
(Proposition \ref{prop3}).

The author is grateful to prof. L. A. Masal'tsev for constant attention to this work. The author
would also thank prof. Yu. A. Nikolayevsky and prof. A. L. Yampolsky for many useful advices
concerning language and style.

\section{Preliminaries} \label{ch1}

Let us recall some basic definitions and facts about the stability of constant mean curvature
hypersurfaces in Riemannian manifolds. Suppose $M$ is a smooth $n$-dimensional manifold immersed in
a smooth $n+1$-dimensional Riemannian manifold as a CMC-hypersurface. Denote by $\eta$ a unit
normal vector field of $M$. Let $D\subset M$ be a compact domain. The {\itshape index form} of $D$
is a quadratic form $Q(\cdot,\cdot)$ on $C^{\infty}(D)$ defined by the equation
\begin{equation}\label{pr4}
Q(w,w)= -\int\limits_D wLw \, dV_M,
\end{equation}
where $dV_M$ is the volume form of the induced metric on $M$, $L$ is the {\itshape Jacobi operator}
$\Delta _M+ \left( \vphantom{2^{2^2}} Ric(\eta,\eta)+\|B\|^2 \right)$, $\ric(\cdot,\cdot)$ is the
Ricci tensor of the ambient manifold, $\|B\|$ is the norm of the second fundamental form of the
immersion, and $\Delta _M $ is the Laplacian of the induced metric (see, for example, \cite{CoM}).

Let $M$ be a minimal hypersurface (a hypersurface of a nonzero constant mean curvature,
respectively). A compact domain $D \subset M$ is called {\itshape stable} if $Q(w,w) \geqslant 0$
for every function $w \in C^{\infty}(D)$ vanishing on $\partial D$ (for every $w \in C^{\infty}(D)$
vanishing on $\partial D$ and with $\int\limits_D w \, dV_M=0$, respectively). The hypersurface $M$
is {\itshape stable} if every compact domain $D\subset M$ is stable, and is {\itshape unstable}
otherwise (see, for example, \cite{BdoCE}). It is proved in \cite[Theorem 1]{FiSc} that if the
{\itshape Jacobi equation} $Lw=0$ admits a solution $w$ strictly positive on $M$, then $M$ is
stable.

Let $(M,g)$ be a smooth Riemannian manifold. Denote by $\Delta_M$
the Laplacian of $g$. For each $\phi \in C^{\infty}(M,S^n)$ denote
by $\Delta_M \phi$ the vector $(\Delta_M \phi_1, \dots,$ $\Delta_M
\phi_{n+1})$, where $(\phi_1, \dots ,\phi_{n+1})$ is the
coordinate functions of $\phi$ for the standard embedding of a
unit sphere $S^n \hookrightarrow \mathbb{R}^{n+1}$. It is well
known that the harmonicity of $\phi$ is equivalent to the equation
$\Delta_M \phi =2e(\phi) \phi$, where $e(\phi)$ is the energy
density function of $\phi$ (see \cite[p.~140, Corollary
(2.24)]{Ur}).

Suppose $M$ is an oriented hypersurface in a $n+1$-dimensional Lie group $N$ with a left invariant
Riemannian metric. Fix the unit normal vector field $\eta$ of $M$ with respect to the orientation.
Let $p$ be a point of $M$. Denote by $L_a$ the left translation by $a \in N$, and let $dL_a$ be the
differential of this map. We can consider $p$ as an element of $N$ if we identify this point with
its image under the immersion. Let $G$ be the map of $M$ to $S^n \subset \mathcal{N}$ such that
$G(p)=\left(dL_p\right)^{-1}\left(\eta(p)\right)$ for all $p \in N$, where $\mathcal{N}$ is the Lie
algebra of $N$. We call $G$ the {\itshape Gauss map} of $M$. It is proved in \cite{ES} that if a
metric of $N$ is bi-invariant (see \cite{Mi} on a structure of such Lie groups), then the Gauss map
is harmonic if and only if the mean curvature of $M$ is constant.

Now we consider the case of nilpotent Lie groups. Let $\mathcal{N}$ be a finite dimensional Lie
algebra over $\mathbb{R}$ with a Lie bracket $[\cdot,\cdot]$. The lower central series of
$\mathcal{N}$ is defined inductively by $\mathcal{N}^1=\mathcal{N}$,
$\mathcal{N}^{k+1}=\left[\mathcal{N}^k,\mathcal{N}\right]$ for all positive integers $k$. The Lie
algebra $\mathcal{N}$ is called {\itshape $k$-step nilpotent} if $\mathcal{N}^k \neq 0$ and
$\mathcal{N}^{k+1}=0$. A Lie group $N$ is called {\itshape $k$-step nilpotent} if its Lie algebra
$\mathcal{N}$ is $k$-step nilpotent.

In the sequel, we consider a $2$-step nilpotent connected and
simply connected Lie group $N$ and its Lie algebra $\mathcal{N}$.
Let $\mathcal{Z}$ be the center of $\mathcal{N}$. Since
$\mathcal{N}$ is $2$-step nilpotent, $0 \neq
[\mathcal{N},\mathcal{N}] \subset \mathcal{Z}$. Suppose that
$\mathcal{N}$ is endowed with a scalar product $\langle \cdot ,
\cdot \rangle$. This scalar product induces a left invariant
Riemannian metric on $N$, which we also denote by $\langle \cdot ,
\cdot \rangle$. Let $\mathcal{V}$ be an orthogonal complement to
$\mathcal{Z}$ in $\mathcal{N}$ with respect to $\langle \cdot ,
\cdot \rangle$. Then
$[\mathcal{V},\mathcal{V}]=[\mathcal{N},\mathcal{N}] \subset
\mathcal{Z}$. For each $Z \in \mathcal{Z}$ a linear operator $J(Z)
\colon \mathcal{V} \rightarrow \mathcal{V}$ is well defined by
$\langle J(Z)X, Y \rangle = \langle [X,Y],Z \rangle$, where $X,Y
\in \mathcal{V}$ are arbitrary vectors.

An important class of $2$-step nilpotent groups consists of so-called $2m+1$-dimensional {\itshape
Heisenberg groups}, which appear in some problems of quantum and Hamiltonian mechanics \cite{Fo}.
The Lie algebra of a Heisenberg group has a basis $K_1, \dots, K_m$, $L_1, \dots, L_m$, $Z$ and the
structure relations
$$
[K_i,L_j]=\delta_{ij}Z, \, [K_i,K_j]=[L_i,L_j]=[K_i,Z]=[L_i,Z]=0,
\, 1\leqslant i,j \leqslant m,
$$
where $\delta_{ij}$ is the Kronecker symbol. We introduce a scalar product such that this basis is
orthonormal. The three-dimensional Heisenberg group with a left invariant Riemannian metric is
often denoted by $Nil$ and is a three-dimensional Thurston geometry. A Lie algebra $\mathcal{N}$ is
of {\itshape Heisenberg type} if $J(Z)^2=-\langle Z,Z \rangle \id|_{\mathcal{V}}$, for every $Z \in
\mathcal{Z}$ \cite{E2}. Its Lie group $N$ is called a Lie group of {\itshape Heisenberg type}. This
class of groups contains, for example, Heisenberg groups and quaternionic Heisenberg groups
\cite[p.~617]{E1}. A general approach to the structure of $2$-step nilpotent Lie algebras was
developed in the paper \cite{E3}.

The Riemannian connection associated with $\langle \cdot,\cdot
\rangle$ is defined on left invariant fields by (see \cite{E1})
\begin{equation}\label{pr1}
\begin{array}{lcll}
\nabla_{X}Y&=&\frac{1}{2}[X,Y], & X,Y \in \mathcal{V};\\
\nabla_{X}Z=\nabla_{Z}X&=&-\frac{1}{2}J(Z)X, & X \in \mathcal{V}, \, Z \in \mathcal{Z};\\
\nabla_{Z}Z^*&=&0, & Z,Z^* \in \mathcal{Z}.\\
\end{array}
\end{equation}
From this one can obtain for the curvature tensor
\begin{equation}\label{pr2}
\begin{array}{lcll}

\begin{array}{l}
R(X,Y)X^*\\
\\
\\
\end{array}
&
\begin{array}{c}
=\\
\\
\\
\end{array}
&
\begin{array}{l}
\frac{1}{2}J([X,Y])X^*\\
-\frac{1}{4}J([Y,X^*])X\\
+\frac{1}{4}J([X,X^*])Y,\\
\end{array}
& X,X^*,Y \in \mathcal{V};\\

\begin{array}{l}
R(X,Z)Y\\
R(X,Y)Z\\
\\
\end{array}
&
\begin{array}{c}
=\\
=\\
\\
\end{array}
&
\begin{array}{l}
-\frac{1}{4}[X,J(Z)Y],\\
-\frac{1}{4}[X,J(Z)Y]\\
+\frac{1}{4}[Y,J(Z)X],\\
\end{array}
& X,Y \in \mathcal{V}, \, Z \in \mathcal{Z};\\

\begin{array}{l}
R(X,Z)Z^*\\
R(Z,Z^*)X\\
\\
\end{array}
&
\begin{array}{c}
=\\
=\\
\\
\end{array}
&
\begin{array}{l}
-\frac{1}{4}J(Z)J(Z^*)X,\\
-\frac{1}{4}J(Z^*)J(Z)X\\
+\frac{1}{4}J(Z)J(Z^*)X,\\
\end{array}
& X \in \mathcal{V}, \, Z,Z^* \in \mathcal{Z};\\

R(Z,Z^*)Z^{**} &=& 0, & Z,Z^*,Z^{**} \in \mathcal{Z}.\\

\end{array}
\end{equation}
And the Ricci tensor is defined by
\begin{equation}\label{pr3}
\begin{array}{lcll}
\ric(X,Y)&=&\frac{1}{2}\sum\limits_{k=1}^{l}\langle J(Z_k)^2 X,Y
\rangle,
 & X,Y \in \mathcal{V};\\
\ric(X,Z)&=&0, & X \in \mathcal{V}, \, Z \in \mathcal{Z};\\
\ric(Z,Z^*)&=&-\frac{1}{4}\trace(J(Z)J(Z^*)), & Z,Z^* \in \mathcal{Z}.\\
\end{array}
\end{equation}
Here $\dim \mathcal{Z}=l$, and $Z_1, \dots, Z_l$ is an orthonormal
basis for $\mathcal{Z}$.

\section{The Laplacian of the Gauss map} \label{ch2-1}

Suppose $\dim N=\dim \mathcal{N}=n+1$, $\dim \mathcal{Z}=n-q+1$,
where $n$ and $q$ are positive integers, $q \leqslant n$.

Let $M$ be a smooth oriented manifold, $\dim M = n$. Suppose $M
\rightarrow N$ is an immersion of this manifold in $N$ as a
hypersurface, and $\eta$ is the unit normal vector field of $M$ in
$N$. For each point $p$ of $M$, suppose that
$\eta(p)=Y_{n+1}=X_{n+1}+Z_{n+1}$, where $X_{n+1} \in
\mathcal{V}$, $Z_{n+1} \in \mathcal{Z}$. Throughout this paper, we
denote by $X_i, \, Y_i, \, Z_i$ elements of $T_p N$ as well as the
corresponding left invariant vector fields, which are elements of
$\mathcal{N}$. Choose an orthonormal frame $\{Y_1, \dots , Y_n\}$
in the vector space $T_p M \subset T_p N$ such that for
$1\leqslant i \leqslant q-1$ $Y_i=X_i$, $Y_q=X_q-Z_q$, and for
$q+1\leqslant i \leqslant n$ $Y_i=Z_i$, where $X_1, \dots, X_q$
are elements of $\mathcal{V}$, $Z_q, \dots, Z_n$ belong to
$\mathcal{Z}$, $X_{n+1}=\lambda X_q$, $Z_{n+1}=\mu Z_q$, where
$\lambda \geqslant 0$ and $\mu \geqslant 0$, $\left| X_q
\right|=\left| Z_{n+1} \right|$, $\left| Z_{q} \right|=\left|
X_{n+1} \right|$. Let $E_1, \ldots E_n$ be an orthonormal frame
defined on some neighborhood $U$ of $p$ such that $E_i(p) = Y_i$
and $\left(\nabla _{E_i}E_j\right)^T(p)=0$, for all $i, j = 1,
\ldots n$ (such a frame is called geodesic at $p$). Here we denote
by $(\cdot)^T$ the projection to $T_p M$.

We can rewrite \eqref{pr3} in the following form
\begin{equation}\label{eq1-ab0}
\begin{array}{lcll}
\ric(X,Y)&=&\frac{1}{2}\sum\limits_{k=q}^{n+1}\langle J(Z_k)^2 X,Y
\rangle,
 & X,Y \in \mathcal{V};\\
\ric(X,Z)&=&0, & X \in \mathcal{V}, \, Z \in \mathcal{Z};\\
\ric(Z,Z^*)&=&-\frac{1}{4}\sum\limits_{1 \leqslant k \leqslant q,
\, k=n+1}
\langle J(Z)J(Z^*)X_k,X_k \rangle, & Z,Z^* \in \mathcal{Z}.\\
\end{array}
\end{equation}
In particular, for all $X,Y \in \mathcal{V}$
\begin{equation}\label{eq1-ab0-1}
\begin{array}{c}
\sum\limits_{1 \leqslant i \leqslant q, \, i=n+1} \langle
J([X,X_i])X_i,Y \rangle\\

=\sum\limits_{1 \leqslant i \leqslant q, \, i=n+1}
\sum\limits_{j=q}^{n+1} \langle [X,X_i],Z_j \rangle \langle
[X_i,Y],Z_j \rangle\\

= -\sum\limits_{j=q}^{n+1} \sum\limits_{1 \leqslant i \leqslant q,
\, i=n+1} \langle J(Z_j)X,X_i \rangle \langle J(Z_j)Y,X_i \rangle\\

=\sum\limits_{j=q}^{n+1}\langle J(Z_j)^2 X, Y \rangle=2\ric(X,Y).\\
\end{array}
\end{equation}

For $1 \leqslant i,j \leqslant n$, denote by $b_{ij}=\langle
\nabla_{E_i}E_j, \eta \rangle$ the coefficients of the second
fundamental form of the immersion, by $\|B\|$ the norm of this
form, and by $H$ the mean curvature of the immersion on $U$. Since
the frame is orthonormal over $U$,
\begin{equation}\label{eq1-ab2}
\begin{array}{c}
H=\frac{1}{n}\sum\limits_{i=1}^{n}b_{ii},\\

\|B\|^2 =\sum\limits_{1 \leqslant i,j \leqslant
n}\left(b_{ij}\right)^2.\\
\end{array}
\end{equation}
Suppose that on $U$
$$
\eta = \sum \limits_{j=1}^{n+1} a_j Y_j,
$$
where $\{a_j\}_{j=1}^{n+1}$ are some functions on $U$. It is clear
that $a_j(p)=\delta_{j \, n+1}$. Then the Gauss map $G \colon U
\rightarrow S^n \subset \mathbb{R}^{n+1}$ takes the form
$$
G = \sum \limits_{j=1}^{n+1} a_j Y_j(e).
$$
In particular, $G(p)=Y_{n+1}(e)$. Denote by $\Delta$ the Laplacian
$\Delta_M$ of the induced metric on $M$.

\begin{theorem}\label{th1}

Let $M$ be a smooth oriented manifold immersed in a 2-step
nilpotent Lie group $N$ as a hypersurface and $G$ be the Gauss map
of $M$. Then, in the above notation

\begin{equation}\label{eq1}
\begin{array}{c}
\Delta G(p) = \sum \limits_{k=1}^q\left( -Y_k(nH)+\sum
\limits_{j=1}^{q-1}\langle J([X_k,X_j])X_j, X_{n+1} \rangle
\right.\\

\vphantom{\sum\limits_{i=1}^{n}} +4\langle
R(X_k,Z_{n+1})Z_{n+1},X_{n+1}
 \rangle - 2\sum \limits_{i=1}^{q}
\sum \limits_{j=q+1}^{n}b_{ij}(p)\langle J(Z_j)X_i, X_k \rangle\\

\left. +2\sum \limits_{i=1}^{q}b_{iq}(p)\langle J(Z_q)X_i, X_k
\rangle+nH(p)\langle
J(Z_{n+1})X_{n+1}, X_k \rangle \right)Y_k(e) \\

+\sum \limits_{k=q+1}^n\left( \vphantom{\sum\limits_{i=1}^{n}} -Y_k(nH) \right)Y_k(e)\\

+\left(\sum \limits_{j=1}^{q-1}\langle J([X_{n+1},X_j])X_j,
X_{n+1} \rangle +4\langle R(X_{n+1},Z_{n+1})Z_{n+1},X_{n+1}  \rangle \right. \\

\vphantom{\sum\limits_{i=1}^{n}} -2\sum \limits_{i=1}^{q} \sum
\limits_{j=q+1}^{n}b_{ij}(p)\langle J(Z_j)X_i, X_{n+1}
\rangle+2\sum \limits_{i=1}^{q}b_{iq}(p)\langle J(Z_q)X_i, X_{n+1}
\rangle\\

\left. \vphantom{\sum\limits_{i=1}^{n}} -\|B\|^2(p)-\ric(Y_{n+1},Y_{n+1}) \right)Y_{n+1}(e). \\
\end{array}
\end{equation}

Here $Y_k(nH)$ denotes the derivative of the function $nH$ with
respect to the vector field $Y_k$.

\end{theorem}

\begin{proof}
Since the frame $E_1, \dots , E_n$ is geodesic at $p$, the
Laplacian at this point has the form
\begin{equation}\label{eq1-ab1}
\Delta G(p) = \sum \limits_{j=1}^{n+1} \sum \limits_{i=1}^{n} E_i
E_i(a_j) Y_j(e).
\end{equation}

For $1 \leqslant i \leqslant n$ we have on $U$
\begin{equation}\label{eq1-pr1}
\nabla_{E_i}\eta=\sum\limits_{j=1}^{n+1}E_i(a_j)Y_j+\sum\limits_{j=1}^{n+1}a_j\nabla_{E_i}Y_j,
\end{equation}
\begin{equation}\label{eq1-pr2}
\nabla_{E_i}\nabla_{E_i}\eta=\sum\limits_{j=1}^{n+1}E_iE_i(a_j)Y_j+2\sum\limits_{j=1}^{n+1}E_i(a_j)\nabla_{E_i}Y_j+
\sum\limits_{j=1}^{n+1}a_j\nabla_{E_i}\nabla_{E_i}Y_j.
\end{equation}
Considering this expression at $p$ and taking its scalar product
with $Y_k$ for $1\leqslant k \leqslant n+1$, we get
$$
\langle \nabla_{E_i}\nabla_{E_i}\eta, Y_k
\rangle=E_iE_i(a_k)+2\sum\limits_{j=1}^{n+1}E_i(a_j)\langle\nabla_{E_i}Y_j,
Y_k \rangle+ \langle \nabla_{E_i}\nabla_{E_i}Y_{n+1}, Y_k \rangle.
$$
Then for scalar coefficients in \eqref{eq1-ab1} we have
\begin{equation}\label{eq1-pr3}
\begin{array}{c}
\sum \limits_{i=1}^{n} E_i E_i(a_k)=\sum\limits_{i=1}^{n}\langle
\nabla_{E_i}\nabla_{E_i}\eta, Y_k \rangle\\

-2\sum\limits_{j=1}^{n+1}\sum
\limits_{i=1}^{n}E_i(a_j)\langle\nabla_{E_i}Y_j, Y_k \rangle-\sum
\limits_{i=1}^{n}\langle \nabla_{E_i}\nabla_{E_i}Y_{n+1}, Y_k
\rangle.
\end{array}
\end{equation}

For $1\leqslant k \leqslant n$, the first expression in
\eqref{eq1-ab2} and the definition of a second fundamental form
imply at $p$
$$
Y_k(nH)=E_k\left(\sum\limits_{i=1}^{n}\langle\nabla_{E_i}E_i,
\eta\rangle\right)=\sum\limits_{i=1}^{n}\langle\nabla_{E_k}\nabla_{E_i}E_i,
\eta\rangle
$$
$$
=\sum\limits_{i=1}^{n} \left( \vphantom{\sum\limits_{i=1}^{n}}
\langle R(E_k,E_i)E_i,
\eta\rangle+\langle\nabla_{E_i}\nabla_{E_k}E_i,
\eta\rangle+\langle\nabla_{[E_k,E_i]}E_i, \eta\rangle \right)
$$
$$
=\sum\limits_{i=1}^{n} \left( \vphantom{\sum\limits_{i=1}^{n}}
\langle R(Y_k,Y_i)Y_i,
Y_{n+1}\rangle+\langle\nabla_{E_i}\nabla_{E_k}E_i, \eta\rangle
\right).
$$
The second equality in the equation above follows from the fact that the projection
$\left(\nabla_{E_i}E_i\right)^T=0$ at $p$ and the vector $\nabla_{E_k}\eta$ is tangent to $M$. The
fourth equality is a consequence of
$$
[E_k,E_i]=\left([E_k,E_i]\right)^{T}=\left(\nabla_{E_k}E_i-\nabla_{E_i}E_k\right)^{T}=0
$$
at $p$. Since $\langle[E_k,E_i],\eta\rangle=0$ on $U$, and
$[E_k,E_i](p)=0$, at $p$ we have
$$
0=\langle\nabla_{E_i}[E_k,E_i],\eta\rangle=\langle\nabla_{E_i}\nabla_{E_k}E_i-\nabla_{E_k}\nabla_{E_i}E_i,\eta\rangle,
$$
for $1\leqslant i \leqslant n$, hence
\begin{equation} \label{eq1-pr4}
Y_k(nH)=\sum\limits_{i=1}^{n} \left(
\vphantom{\sum\limits_{i=1}^{n}} \langle R(Y_k,Y_i)Y_i,
Y_{n+1}\rangle+\langle\nabla_{E_i}\nabla_{E_i}E_k, \eta\rangle
\right).
\end{equation}

Differentiating $\langle E_k, \eta\rangle=0$ two times with
respect to $E_i$ (here we put $1\leqslant k,i \leqslant n$) and
using $\langle \nabla_{E_i}E_k, \nabla_{E_i}\eta\rangle(p)=0$, we
derive from \eqref{eq1-pr4}
\begin{equation}\label{eq1-pr5}
\begin{array}{c}
\sum\limits_{i=1}^{n}\langle \nabla_{E_i}\nabla_{E_i}\eta, Y_k
\rangle= -\sum\limits_{i=1}^{n}\langle
\nabla_{E_i}\nabla_{E_i}E_k, \eta \rangle\\

=-Y_k(nH)+\sum\limits_{i=1}^{n}\langle R(Y_k,Y_i)Y_i,
Y_{n+1}\rangle=-Y_k(nH)+\ric(Y_k,Y_{n+1}).
\end{array}
\end{equation}
For $1 \leqslant i \leqslant n$, differentiating $\langle
\eta,\eta \rangle=1$ two times with respect to $E_i$, we get
$2\langle \nabla_{E_i}\nabla_{E_i}\eta,\eta \rangle+2\langle
\nabla_{E_i}\eta,\nabla_{E_i}\eta \rangle=0$. This equation and
the second expression in \eqref{eq1-ab2} imply at $p$
\begin{equation}\label{eq1-pr6}
\begin{array}{c}
\sum\limits_{i=1}^{n}\langle \nabla_{E_i}\nabla_{E_i}\eta, Y_{n+1}
\rangle= -\sum\limits_{i=1}^{n}\langle
\nabla_{E_i}\eta, \nabla_{E_i}\eta\rangle\\

=-\sum\limits_{i=1}^{n}\sum\limits_{j=1}^{n}\langle
\nabla_{E_i}\eta,E_j\rangle \langle \nabla_{E_i}\eta,E_j\rangle\\

\vphantom{\sum\limits_{i=1}^{n}} =-\sum\limits_{1 \leqslant i,j
\leqslant n}\langle \nabla_{E_i}E_j, \eta\rangle^2 = -\|B\|^2(p).
\end{array}
\end{equation}

Consider the scalar products $\langle \nabla_{E_i} \eta , Y_{k}
\rangle$ at the point $p$, for $1\leqslant i \leqslant n$,
$1\leqslant k \leqslant n+1$. As $|\eta| = |Y_{n+1}| = 1$, we
obtain from \eqref{eq1-pr1}
$$
0=\langle \nabla_{E_i}\eta, \eta \rangle(p)=\langle
\nabla_{E_i}\eta, Y_{n+1} \rangle=E_i(a_{n+1})+\langle
\nabla_{E_i}Y_{n+1}, Y_{n+1} \rangle=E_i(a_{n+1}).
$$
For $1\leqslant k \leqslant n$, $\langle E_k, \eta \rangle=0$ imply

$$
\begin{array}{c}
b_{ik}(p)=\langle \nabla_{E_i}E_k, \eta \rangle(p)=-\langle \nabla_{E_i}\eta, E_k
\rangle(p)=-\langle \nabla_{E_i}\eta, Y_k \rangle\\
=-E_i(a_k)-\langle \nabla_{E_i}Y_{n+1}, Y_k \rangle.\\
\end{array}
$$

Hence at $p$ we have
$$
-2\sum\limits_{j=1}^{n+1}\sum
\limits_{i=1}^{n}E_i(a_j)\langle\nabla_{E_i}Y_j, Y_k
\rangle
$$
$$
=2\sum\limits_{j=1}^{n}\sum \limits_{i=1}^{n}\left( \vphantom{\sum\limits_{i=1}^{n}}
 b_{ij}(p)+\langle \nabla_{Y_i}Y_{n+1}, Y_j
\rangle\right)\langle\nabla_{Y_i}Y_j, Y_k \rangle.
$$
It follows from \eqref{pr1} that for $1\leqslant i,j \leqslant q$
the expression $b_{ij}(p)\langle \nabla_{X_i}X_j, Y_k \rangle$ is
skew-symmetric with respect to $i,j$; hence the sum of such terms
with respect to $i$ and $j$ vanishes. Sum up other expressions
using the symmetry $\nabla_{X}Z=\nabla_{Z}X$ for all $X \in
\mathcal{V}$, $Z \in \mathcal{Z}$ and the symmetry of the second
fundamental form. We obtain
\begin{equation}\label{eq1-pr7}
\begin{array}{c}
-2\sum\limits_{j=1}^{n+1}\sum
\limits_{i=1}^{n}E_i(a_j)\langle\nabla_{E_i}Y_j, Y_k
\rangle=-2\sum\limits_{i=1}^{q}\sum \limits_{j=q+1}^{n}
b_{ij}(p)\langle J(Z_j)X_i, Y_k \rangle\\

+2\sum\limits_{i=1}^{q} b_{iq}(p)\langle J(Z_q)X_i, Y_k
\rangle+2\sum\limits_{1 \leqslant i,j \leqslant n} \langle
\nabla_{Y_i}Y_{n+1}, Y_j \rangle\langle\nabla_{Y_i}Y_j, Y_k
\rangle.\\
\end{array}
\end{equation}
Now we can complete the proof of the theorem using the following
technical lemmas.

\begin{lemma}\label{lem1}

The last summand on the right hand side of \eqref{eq1-pr7} is
equal to

\begin{equation}\label{eq1-pr8}
\begin{array}{c}
2\sum\limits_{1 \leqslant i,j \leqslant n}\langle
\nabla_{Y_i}Y_{n+1}, Y_j \rangle\langle\nabla_{Y_i}Y_j, Y_k
\rangle\\

=\left[
\begin{array}{ll}
2\ric(Y_k,Y_{n+1})+4\langle R(X_k,Z_{n+1})Z_{n+1},X_{n+1}\rangle,
& 1 \leqslant k \leqslant q-1;\\

\begin{array}{l}
2\ric(X_q,X_{n+1})+2\ric(Z_q,Z_{n+1})\\

+4\langle R(X_q,Z_{n+1})Z_{n+1},X_{n+1}\rangle\\

-4\langle R(X_{n+1},Z_q)Z_{n+1},X_{n+1}\rangle,\\
\end{array}

& k=q;\\

\begin{array}{l}
-2\ric(Y_k,Y_{n+1})\\

+4\langle R(X_{n+1},Z_k)Z_{n+1},X_{n+1}\rangle,\\
\end{array}

& q+1 \leqslant k \leqslant n;\\

\begin{array}{l}
2\ric(X_{n+1},X_{n+1})-2\ric(Z_{n+1},Z_{n+1})\\

+8\langle R(X_{n+1},Z_{n+1})Z_{n+1},X_{n+1}\rangle,\\
\end{array}

& k=n+1.\\
\end{array}
\right.\\
\end{array}
\end{equation}

\end{lemma}

\begin{lemma}\label{lem2}

The last summand on the right hand side of \eqref{eq1-pr3} can be
reduced to the form

\begin{equation}\label{eq1-pr15}
\begin{array}{c}
-\sum \limits_{i=1}^{n}\langle \nabla_{E_i}\nabla_{E_i}Y_{n+1},
Y_k \rangle\\

=\left[
\begin{array}{ll}
nH(p)\langle J(Z_{n+1})X_{n+1}, X_k \rangle-\ric(Y_k,Y_{n+1}),
& 1 \leqslant k \leqslant q-1;\\

\begin{array}{l}
-\ric(X_q,X_{n+1})-\ric(Z_q,Z_{n+1})\\

+4\langle R(X_{n+1},Z_q)Z_{n+1},X_{n+1} \rangle,\\
\end{array}

& k=q;\\

\ric(Y_k,Y_{n+1})-4\langle R(X_{n+1},Z_k)Z_{n+1},X_{n+1} \rangle,
& q+1 \leqslant k \leqslant n;\\

\begin{array}{l}
-\ric(X_{n+1},X_{n+1})+\ric(Z_{n+1},Z_{n+1})\\

-4\langle R(X_{n+1},Z_{n+1})Z_{n+1},X_{n+1} \rangle,\\
\end{array}

& k=n+1.\\
\end{array}
\right.\\
\end{array}
\end{equation}

\end{lemma}

Now, if we combine \eqref{eq1-pr3} with \eqref{eq1-pr7},
\eqref{eq1-pr8}, \eqref{eq1-pr15}, and \eqref{eq1-ab0-1}, we get
\eqref{eq1}.
\end{proof}

\begin{proof}[\rm{P r o o f \; o f \; L e m m a \, \ref{lem1}}]

For $1\leqslant k \leqslant q-1$ from the expressions for the
Riemannian connection we get
$$
\sum\limits_{i=1}^{n}\sum \limits_{j=1}^{n}\langle
\nabla_{Y_i}Y_{n+1}, Y_j \rangle\langle\nabla_{Y_i}Y_j, X_k
\rangle=\sum\limits_{i=1}^{q-1}\left(
\vphantom{\sum\limits_{i=1}^{n}} \langle \frac{1}{2}[X_i,X_{n+1}],
-Z_q \rangle\right.
$$
$$
+\left. \vphantom{\sum\limits_{i=1}^{n}} \langle
-\frac{1}{2}J(Z_{n+1})X_i, X_q \rangle\right)\langle
\frac{1}{2}J(Z_q)X_i, X_k \rangle
$$
$$
+\sum\limits_{i=1}^{q-1}\sum \limits_{j=q+1}^{n}\langle \frac{1}{2}[X_i,X_{n+1}], Z_j \rangle
\langle -\frac{1}{2}J(Z_j)X_i, X_k \rangle
$$
$$
+\sum \limits_{j=1}^{q-1}\langle -\frac{1}{2}J(Z_{n+1})X_q+\frac{1}{2}J(Z_{q})X_{n+1}, X_j
\rangle\langle \frac{1}{2}J(Z_q)X_j, X_k \rangle
$$
$$
+\left( \vphantom{\sum\limits_{i=1}^{n}} \langle
\frac{1}{2}J(Z_q)X_{n+1}-\frac{1}{2}J(Z_{n+1})X_q,
X_q\rangle+\langle \frac{1}{2}[X_q,X_{n+1}], -Z_q \rangle\right)
\langle J(Z_q)X_q, X_k \rangle
$$
$$
+\sum \limits_{j=q+1}^{n}\langle \frac{1}{2}[X_q,X_{n+1}], Z_j
\rangle\langle -\frac{1}{2} J(Z_j)X_q, X_k
\rangle
$$
$$
+\sum\limits_{i=q+1}^{n}\sum \limits_{j=1}^{q-1}\langle -\frac{1}{2}J(Z_i)X_{n+1}, X_j \rangle
\langle -\frac{1}{2}J(Z_i)X_j, X_k \rangle
$$
$$
+\sum\limits_{i=q+1}^{n}\langle -\frac{1}{2}J(Z_i)X_{n+1}, X_q \rangle\langle
-\frac{1}{2}J(Z_i)X_q, X_k \rangle.
$$
The skew-symmetry of $J$ implies $\langle J(Z)X_q,X_{n+1}\rangle=-\langle J(Z)X_{n+1},X_q\rangle=0$
for all $Z \in \mathcal{Z}$, and $[X_q,X_{n+1}]=0$. Hence we can rewrite the above expression in
the form
$$
\sum\limits_{i=1}^{n}\sum \limits_{j=1}^{n}\langle
\nabla_{Y_i}Y_{n+1}, Y_j \rangle\langle\nabla_{Y_i}Y_j, X_k
\rangle
$$
$$
=\frac{1}{2}\sum \limits_{j=q}^{n}\sum\limits_{1 \leqslant i
\leqslant q , \, i=n+1}\langle -J(Z_j)X_{n+1}, X_i \rangle \langle
J(Z_j)X_k, X_i \rangle
$$
$$
=-\frac{1}{2}\sum \limits_{j=q}^{n}\langle J(Z_j)X_{n+1},
J(Z_j)X_k \rangle=\frac{1}{2}\sum \limits_{j=q}^{n}\langle
J(Z_j)^2X_{n+1}, X_k \rangle
$$
$$
=\ric(X_k,X_{n+1})+2\langle R(X_k,Z_{n+1})Z_{n+1},X_{n+1}\rangle.
$$
This implies the first equality in \eqref{eq1-pr8}.

For $q+1 \leqslant k \leqslant n$ we have
$$
\sum\limits_{i=1}^{n}\sum \limits_{j=1}^{n}\langle
\nabla_{Y_i}Y_{n+1}, Y_j \rangle\langle\nabla_{Y_i}Y_j, Z_k
\rangle=\sum\limits_{i=1}^{q-1}\sum \limits_{j=1}^{q-1}\langle
-\frac{1}{2}J(Z_{n+1})X_i,X_j \rangle \langle
\frac{1}{2}[X_i,X_j], Z_k \rangle
$$
$$
+\frac{1}{2} \sum \limits_{i=1}^{q-1}\left(
\vphantom{\sum\limits_{i=1}^{n}} \langle -\frac{1}{2}J(Z_{n+1})X_i
,X_q\rangle+\langle \frac{1}{2}[X_i,X_{n+1}], -Z_q \rangle \right)
\langle \frac{1}{2}[X_i,X_q], Z_k \rangle
$$
$$
+\sum \limits_{j=1}^{q-1}\langle
-\frac{1}{2}J(Z_{n+1})X_q+\frac{1}{2}J(Z_q)X_{n+1}, X_j
\rangle\langle \frac{1}{2} [X_q,X_j], Z_k \rangle
$$
$$
=\frac{1}{4}\sum\limits_{i=1}^{q}\sum \limits_{j=1}^{q}\langle
-J(Z_{n+1})X_i,X_j \rangle \langle [X_i,X_j], Z_k \rangle
$$
$$
=\frac{1}{4}\sum\limits_{1 \leqslant i \leqslant q, \, i=n+1}\sum
\limits_{1 \leqslant j \leqslant q, \, j=n+1}\langle-
J(Z_{n+1})X_i ,X_j\rangle \langle J(Z_k) X_i,X_j\rangle
$$
$$
-\frac{1}{2} \sum \limits_{1 \leqslant i \leqslant q, \,
i=n+1}\langle -J(Z_{n+1})X_{n+1},X_i\rangle \langle
J(Z_k)X_{n+1},X_i\rangle
$$
$$
=\frac{1}{4}\sum\limits_{1 \leqslant i \leqslant q, \,
i=n+1}\langle- J(Z_{n+1})X_i ,J(Z_k) X_i\rangle-\frac{1}{2}
\langle -J(Z_{n+1})X_{n+1},J(Z_k)X_{n+1}\rangle
$$
$$
=\frac{1}{4}\sum\limits_{1 \leqslant i \leqslant q, \,
i=n+1}\langle J(Z_k)J(Z_{n+1})X_i , X_i\rangle-\frac{1}{2} \langle
J(Z_k)J(Z_{n+1})X_{n+1},X_{n+1}\rangle
$$
$$
=-\ric(Z_k,Z_{n+1})+2\langle R(X_{n+1},Z_k)Z_{n+1},X_{n+1}\rangle.
$$
This completes the proof of \eqref{eq1-pr8} and of the lemma, as
$Y_q=X_q-Z_q$, and $Y_{n+1}=X_{n+1}+Z_{n+1}$.
\end{proof}

\begin{proof}[\rm{P r o o f \; o f \; L e m m a  \,  \ref{lem2}}]
Let on $U$
\begin{equation}\label{eq1-pr11-1}
E_i=\sum \limits_{j=1}^{n+1}c_{ij} Y_j,
\end{equation}
where $c_{ij}$, $1\leqslant i \leqslant n$, $1\leqslant j
\leqslant n+1$ are scalar functions on $U$. Note that
$E_i(p)=Y_i$, so $c_{ij}(p)=\delta_{ij}$. Using
\eqref{eq1-pr11-1}, we get
\begin{equation}\label{eq1-pr12}
\begin{array}{c}
 \nabla_{E_i}Y_{n+1}=\sum \limits_{j=1}^{n+1}c_{ij}
\nabla_{Y_j}Y_{n+1}=\frac{1}{2}\sum \limits_{j=1}^{q}c_{ij} \left(
\vphantom{\sum\limits_{i=1}^{n}} [X_j,X_{n+1}]-J(Z_{n+1})X_j
\right) \\

+\frac{1}{2}c_{iq}J(Z_q)X_{n+1}-\frac{1}{2}\sum
\limits_{j=q+1}^{n}c_{ij} J(Z_j)X_{n+1}-c_{i \, n+1}J(Z_{n+1})X_{n+1}.\\
\end{array}
\end{equation}
Also, for $1 \leqslant k \leqslant n$ at $p$ we have
\begin{equation}\label{eq1-pr13}
\nabla_{E_k}E_i=\sum \limits_{j=1}^{n+1}\left(
\vphantom{\sum\limits_{i=1}^{n}} Y_k(c_{ij})
Y_j+c_{ij}(p)\nabla_{Y_k}Y_j\right)=\sum
\limits_{j=1}^{n+1}Y_k(c_{ij}) Y_j + \nabla_{Y_k}Y_i.
\end{equation}
In particular, at $p$
\begin{equation}\label{eq1-pr14}
b_{ki}(p)=\langle \nabla_{E_k}E_i, \eta \rangle(p)=Y_k(c_{i \,
n+1}) + \langle \nabla_{Y_k}Y_i, Y_{n+1} \rangle.
\end{equation}
Considering \eqref{eq1-pr13} for $k=i$, projecting both sides of
it to $T_p M$, and using the properties of the geodesic frame, we
get
$$
0=\sum\limits_{j=1}^{n}Y_i(c_{ij})Y_j+\left(\nabla_{Y_i}Y_i\right)^T.
$$
For $1 \leqslant i \leqslant q-1$ and $q+1 \leqslant i \leqslant
n$ $\nabla_{Y_i}Y_i=0$, and
$\nabla_{Y_q}Y_q=J(Z_q)X_q=\left(\nabla_{Y_q}Y_q\right)^T$, since
$\langle J(Z_q)X_q, X_{n+1}+Z_{n+1}\rangle=0$. Then, for $1
\leqslant j \leqslant q$ we obtain
$$
Y_i(c_{ij})= \left[
\begin{array}{ll}
0, & 1\leqslant i \leqslant q-1;\\
-\langle J(Z_q)X_q, Y_j \rangle, & i=q;\\
0, & q+1\leqslant i \leqslant n.\\
\end{array}
\right.
$$
We can deduce from \eqref{eq1-pr14} and the above considerations
that for $1 \leqslant i \leqslant n$ $b_{ii}(p)=Y_i(c_{i \,
n+1})$. Differentiate \eqref{eq1-pr12} with respect to $E_i$ at
$p$. For $1 \leqslant i \leqslant q-1$ we get
$$
\nabla_{E_i}\nabla_{E_i}Y_{n+1}=-Y_i(c_{i \,
n+1})J(Z_{n+1})X_{n+1}+\frac{1}{2} \nabla_{X_i} \left(
\vphantom{\sum\limits_{i=1}^{n}} [X_i,X_{n+1}]-J(Z_{n+1})X_i
\right)
$$
$$
=-b_{ii}(p)J(Z_{n+1})X_{n+1}-\frac{1}{4}
J([X_i,X_{n+1}])X_i-\frac{1}{4}[X_i, J(Z_{n+1})X_i].
$$
For $i=q$ we have
$$
\nabla_{E_q}\nabla_{E_q}Y_{n+1}=-Y_q(c_{i \,
n+1})J(Z_{n+1})X_{n+1}+\sum\limits_{j=1}^{n}Y_q(c_{qj})\nabla_{Y_j}Y_{n+1}
$$
$$
+\frac{1}{2}\nabla_{Y_q} \left( \vphantom{\sum\limits_{i=1}^{n}}
[X_q,X_{n+1}]-J(Z_{n+1})X_q+J(Z_q)X_{n+1}
\right)=-b_{qq}(p)J(Z_{n+1})X_{n+1}
$$
$$
-\frac{1}{2}\sum\limits_{j=1}^{q-1}\langle J(Z_q)X_q, X_j \rangle
\left( \vphantom{\sum\limits_{i=1}^{n}}
[X_j,X_{n+1}]-J(Z_{n+1})X_j \right)
$$
$$
-\frac{1}{4}[X_q,J(Z_{n+1})X_q]+\frac{1}{4}[X_q,J(Z_q)X_{n+1}]-\frac{1}{4}J(Z_q)
J(Z_{n+1})X_q+\frac{1}{4}J(Z_q)^2X_{n+1}.
$$
For $q+1 \leqslant i \leqslant n$ we obtain
$$
\nabla_{E_i}\nabla_{E_i}Y_{n+1}=-Y_i(c_{i \,
n+1})J(Z_{n+1})X_{n+1}-\frac{1}{2}\nabla_{Z_i}\left(J(Z_i)X_{n+1}\right)
$$
$$
=-b_{ii}(p)J(Z_{n+1})X_{n+1}+\frac{1}{4}J(Z_i)^2 X_{n+1}.
$$
Summing up these expressions, we get for $1 \leqslant k \leqslant q-1$
$$
-\sum \limits_{i=1}^{n}\langle \nabla_{E_i}\nabla_{E_i}Y_{n+1},
X_k \rangle=nH(p)\langle J(Z_{n+1})X_{n+1}, X_k \rangle
$$
$$
+\frac{1}{4}\sum \limits_{i=1}^{q-1}\langle J([X_i,X_{n+1}])X_i,
X_k \rangle+\frac{1}{2}\sum\limits_{j=1}^{q-1}\langle J(Z_q)X_q,
X_j \rangle \langle -J(Z_{n+1})X_j, X_k \rangle
$$
$$
+\frac{1}{4}\langle J(Z_q) J(Z_{n+1})X_q, X_k
\rangle-\frac{1}{4}\langle J(Z_q)^2X_{n+1}, X_k
\rangle-\frac{1}{4}\sum \limits_{i=q+1}^{n} \langle J(Z_i)^2
X_{n+1}, X_k \rangle
$$
$$
=nH(p)\langle J(Z_{n+1})X_{n+1}, X_k \rangle -\frac{1}{2}\sum
\limits_{1 \leqslant i \leqslant q, \, i=n+1}\langle
J([X_{n+1},X_i])X_i, X_k \rangle
$$
$$
=nH(p)\langle J(Z_{n+1})X_{n+1}, X_k \rangle-\ric(X_k,X_{n+1}).
$$
Here we use the equation $J(Z_q)J(Z_{n+1})X_q=J(Z_{n+1})^2X_{n+1}$, which follows from the
construction of the frame. Thus we obtain the first expression in \eqref{eq1-pr15}.

For $q+1 \leqslant k \leqslant n$ we have
$$
-\sum \limits_{i=1}^{n}\langle \nabla_{E_i}\nabla_{E_i}Y_{n+1},
Z_k \rangle= \frac{1}{4}\sum \limits_{i=1}^{q-1}\langle [X_i,
J(Z_{n+1})X_i], Z_k \rangle
$$
$$
+\frac{1}{2}\sum\limits_{j=1}^{q-1}\langle J(Z_q)X_q, X_j \rangle
\langle [X_j,X_{n+1}], Z_k \rangle
$$
$$
+\frac{1}{4}\langle [X_q,J(Z_{n+1})X_q], Z_k
\rangle-\frac{1}{4}\langle [X_q,J(Z_q)X_{n+1}], Z_k \rangle
$$
$$
=-\frac{1}{4}\sum \limits_{1 \leqslant i \leqslant q, \,
i=n+1}\langle J(Z_k)J(Z_{n+1})X_i, X_i \rangle+\langle
J(Z_k)J(Z_{n+1})X_{n+1}, X_{n+1} \rangle
$$
$$
=\ric(Z_k,Z_{n+1})-4\langle R(X_{n+1},Z_k)Z_{n+1},X_{n+1} \rangle.
$$
In the above calculation we used the fact that
$J(Z_q)X_q=J(Z_{n+1})X_{n+1}$ and
$[X_q,J(Z_q)X_{n+1}]=[X_{n+1},J(Z_{n+1})X_{n+1}]$. As
$Y_q=X_q-Z_q$ and $Y_{n+1}=X_{n+1}+Z_{n+1}$, we get the last three
equalities in \eqref{eq1-pr15}.
\end{proof}

\section{Mean curvature and harmonicity} \label{ch2-2}

Consider the tangent bundle $TN$ and the distribution in $TN$ formed by left invariant vector
fields from $\mathcal{Z}$. Since $\mathcal{Z}$ is an abelian ideal, we can integrate this
distribution and obtain a foliation. Denote this foliation by $\mathcal{F}_{\mathcal{Z}}$. Let G be
harmonic. Since by \eqref{eq1}, in this case $Y_k(nH) = 0$ for all $q+1 \leqslant k \leqslant n$,
we have

\begin{corollary}\label{cor1}
If the Gauss map of $M$ is harmonic, then for each leaf $M'$ of
$\mathcal{F}_{\mathcal{Z}}$ the mean curvature of the immersion is
constant on $M \cap M'$.
\end{corollary}

Now we obtain some analogues of the results for Lie groups with
bi-invariant metrics that were stated in \cite{ES}.

Let $\nu$ be a vector field on $M$ defined by $\nu(p)=Y_q$, for $p \in M$. In other words, we
obtain $\nu(p)$ rotating the unit normal vector $\eta(p)$ by the angle $\frac{\pi}{2}$ in the
2-plane containing $\eta(p)$ and orthogonal to both $dL_p(\mathcal{V})$ and $dL_p(\mathcal{Z})$.

\begin{proposition}\label{prop1-1}

Let $M$ be a compact smooth oriented hypersurface in a $2$-step
nilpotent Lie group $N$. Assume that

\begin{enumerate}

\item the mean curvature of $M$ is constant on the integral curves
of $\nu$;

\item the Gauss map of $M$ is harmonic;

\item $\|B\|^2+\ric(\eta,\eta) \geqslant 0$ on $M$ and
$\|B\|^2+\ric(\eta,\eta) > 0$ in some point of $M$;

\item the set of points $p \in M$ such that $\eta(p) \notin
dL_p(\mathcal{V})$ is dense in $M$.

\end{enumerate}

Then $G(M)$ is contained in a closed hemisphere of $S^n$ if and
only if $G(M)$ is contained in a great sphere of $S^n$.

\end{proposition}
\begin{proof}
One of the implications in the proposition is obvious. Suppose that some closed hemisphere of $S^n$
contains $G(M)$, i.e., there exists a unit vector $v \in \mathbb{R}^{n+1}$ such that for all $p \in
M$ $\langle G(p), v \rangle$ is nonpositive. Consider a smooth function $f=\langle G, v \rangle $
on $M$. The coefficient of $Y_q(e)$ in \eqref{eq1} vanishes. For all points from some dense set of
$M$ we have $X_q \neq 0$ and thus $X_{n+1}=\frac{\left| X_{n+1}\right|}{\left| X_q\right|}X_q$.
This, together with $Y_q(nH)=0$, implies that the coefficient of $Y_{n+1}(e)$ is equal to
$-\|B\|^2-\ric(\eta,\eta)$ on the dense subset of $M$ and hence on the whole $M$ because both the
coefficient and $-\|B\|^2-\ric(\eta,\eta)$ are continuous. Taking the scalar product of \eqref{eq1}
with $v$, we obtain
$$
\Delta f = -\left( \vphantom{2^{2^2}} \|B\|^2 +\ric(\eta,\eta)
\right)f \geqslant 0.
$$
Then $f$ is a subharmonic function on the compact manifold $M$. Thus $f$ is constant, and $\left(
\vphantom{2^{2^2}} \|B\|^2 +\ric(\eta,\eta) \right)f=-\Delta f=0$. From the hypothesis, this
implies $f=0$, hence $G(M)$ is contained in the equator $v^{\perp}$. This completes the proof.
\end{proof}

\begin{proposition}\label{prop1-2}
Suppose that a smooth oriented hypersurface $M$ in a $2$-step
nilpotent Lie group $N$ is CMC, its Gauss map is harmonic, for all
$p$ from some dense set of $M$ the normal vector $\eta(p) \notin
dL_p(\mathcal{V})$, and $G(M)$ is contained in an open hemisphere
of $S^n$. Then $M$ is stable.
\end{proposition}

\begin{proof}
From the hypothesis, there exists $v \in \mathbb{R}^{n+1}$ such
that for all $p \in M$ $\langle G(p), v \rangle
>0$. As in the proof of Proposition \ref{prop1-1}, consider a scalar function
$w(p)=\langle G(p), v \rangle $ on $M$. This function is smooth and positive. As above, \eqref{eq1}
implies the Jacobi equation $\left(\Delta +\|B\|^2+\ric(\eta,\eta)\right)w =0$. Now \cite[Theorem
1]{FiSc} implies the stability of $M$.
\end{proof}

\section{Groups of Heisenberg type} \label{ch3}

Let $N$ be a group of Heisenberg type. Then from \eqref{eq1-ab0}, for all $X,Y \in \mathcal{V}$,
$$
\ric(X,Y)=\frac{1}{2}\sum\limits_{k=q}^{n+1}\langle J(Z_k)^2 X, Y
\rangle=-\frac{1}{2}(n+1-q)\langle X,Y \rangle.
$$
Also, we can rewrite the coefficients in \eqref{eq1} for $1
\leqslant k \leqslant q$ and for $k=n+1$ in the form
$$
\sum \limits_{j=1}^{q-1} \langle J([X_k,X_j])X_j, X_{n+1}
\rangle+4\langle R(X_k, Z_{n+1})Z_{n+1}, X_{n+1} \rangle
$$
$$
=\left[
\begin{array}{ll}
0, & 1 \leqslant k \leqslant q-1;\\
\left|Z_{n+1}\right|\left|X_{n+1}\right|\left(q-n-1+\left|Z_{n+1}\right|^2\right), & k=q;\\
\left|X_{n+1}\right|^2\left(q-n-1+\left|Z_{n+1}\right|^2\right), & k=n+1.\\
\end{array}
\right.
$$
Moreover,
$$
\ric(Z_{n+1},Z_{n+1})=-\frac{1}{4}\trace
J(Z_{n+1})^2=\frac{q}{4}\left|Z_{n+1}\right|^2,
$$
and thus
$$
\ric(Y_{n+1},Y_{n+1})=
\frac{q}{4}\left|Z_{n+1}\right|^2-\frac{1}{2}(n+1-q)\left|X_{n+1}\right|^2.
$$
Equation \eqref{eq1} now takes the form
\begin{equation}\label{eq4}
\begin{array}{c}
\Delta G(p) = \sum \limits_{k=1}^{q-1}\left( -Y_k(nH)- 2\sum
\limits_{i=1}^{q} \sum \limits_{j=q+1}^{n}b_{ij}(p)\langle
J(Z_j)X_i, X_k \rangle
\right.\\

\left. \vphantom{\sum\limits_{i=1}^{n}} +2\sum
\limits_{i=1}^{q}b_{iq}(p)\langle J(Z_q)X_i, X_k
\rangle+nH(p)\langle
J(Z_{n+1})X_{n+1}, X_k \rangle \right)Y_k(e) \\

+\left( \vphantom{\sum\limits_{i=1}^{n}}
-Y_q(nH)+\left|Z_{n+1}\right|\left|X_{n+1}\right|\left(q-n-1+\left|Z_{n+1}\right|^2\right)
 \right.\\

\vphantom{\sum\limits_{i=1}^{n}} -2\sum \limits_{i=1}^{q} \sum
\limits_{j=q+1}^{n}b_{ij}(p)\langle
J(Z_j)X_i, X_q \rangle\\

\left. +2\sum \limits_{i=1}^{q}b_{iq}(p)\langle J(Z_q)X_i, X_q
\rangle+nH(p)\langle
J(Z_{n+1})X_{n+1}, X_q \rangle \right)Y_q(e) \\

+\sum \limits_{k=q+1}^n\left( \vphantom{\sum\limits_{i=1}^{n}} -Y_k(nH) \right)Y_k(e)\\

+\left( - 2\sum \limits_{i=1}^{q} \sum
\limits_{j=q+1}^{n}b_{ij}(p)\langle J(Z_j)X_i, X_{n+1}
\rangle\right. \\

\left. \vphantom{\sum\limits_{i=1}^{n}}+2\sum
\limits_{i=1}^{q}b_{iq}(p)\langle J(Z_q)X_i, X_{n+1}
\rangle-\|B\|^2(p)- \frac{q}{4}\left|Z_{n+1}\right|^2\right.\\

\left. \vphantom{\sum\limits_{i=1}^{n}}
+\left|X_{n+1}\right|^2\left(\frac{1}{2}(q-n-1)+\left|Z_{n+1}\right|^2\right)
 \right)Y_{n+1}(e). \\
\end{array}
\end{equation}

Consider the case $n=q$, i.e., $\dim \mathcal{Z}=1$. It is easy to
see that $n$ is then even, $n=2m$, where $m$ is a positive
integer, and $N$ is isomorphic to the $2m+1$-dimensional
Heisenberg group (recall that $N$ is connected and simply
connected).

In this case at $p$ we can choose $X_1, \dots, X_{2m+1}$ so that
$$
J(Z)X_i=X_{m+i}, \, 1 \leqslant i \leqslant m-1;
$$
$$
J(Z)X_m= \frac{X_{2m}}{\left| X_{2m}\right|}=\frac{X_{2m}}{\left|
Z_{2m+1}\right|} \text{ if } Z_{2m+1} \neq 0; \text{or}
\frac{X_{2m+1}}{\left| X_{2m+1}\right|} \text{ if } X_{2m+1} \neq
0;
$$
$$
J(Z)X_{m+i}=-X_i, \, 1 \leqslant i \leqslant m-1;
$$
$$
J(Z)X_{2m}=-\left| X_{2m}\right|X_m=-\left| Z_{2m+1}\right|X_m;
$$
$$
J(Z)X_{2m+1}=-\left| X_{2m+1}\right|X_m.
$$
Choose $Z_{2m}=\left|X_{2m+1}\right|Z$ and
$Z_{2m+1}=\left|Z_{2m+1}\right|Z$. Then \eqref{eq4} has the form
\begin{equation}\label{eq5}
\begin{array}{c}
\Delta G(p) = -\sum \limits_{k=1}^{m-1}\left(
\vphantom{\sum\limits_{i=1}^{n}} Y_k(2mH)+ 2b_{2m \,
m+k}(p)\left|X_{2m+1}\right|\right)Y_k(e) \\

-\left( \vphantom{\sum\limits_{i=1}^{n}} Y_m(2mH)+
2mH(p)\left|X_{2m+1}\right|\right.\\

\left. \vphantom{\sum\limits_{i=1}^{n}} +2b_{2m \,
2m}(p)\left|X_{2m+1}\right|\left|Z_{2m+1}\right|\right)Y_m(e) \\

-\sum \limits_{k=1}^{m-1}\left( \vphantom{\sum\limits_{i=1}^{n}}
Y_{m+k}(2mH)- 2b_{2m \,
k}(p)\left|X_{2m+1}\right|\right)Y_k(e) \\

-\left( \vphantom{\sum\limits_{i=1}^{n}} Y_{2m}(2mH)+
\left|X_{2m+1}\right|^3\left|Z_{2m+1}\right|
\right.\\

\left. \vphantom{\sum\limits_{i=1}^{n}} -2b_{2m \,
m}(p)\left|X_{2m+1}\right|\left|Z_{2m+1}\right|\right)Y_{2m}(e) \\

-\left( \vphantom{\sum\limits_{i=1}^{n}}
\|B\|^2(p)+\frac{m}{2}\left|Z_{2m+1}\right|^2-
\frac{1}{2}\left|X_{2m+1}\right|^2  \right.\\

\left. \vphantom{\sum\limits_{i=1}^{n}} +\left|X_{2m+1}\right|^4-
2b_{2m \,
m}(p)\left|X_{2m+1}\right|^2\right)Y_{2m+1}(e). \\
\end{array}
\end{equation}

Consider an example of the three-dimensional Heisenberg group $Nil$. In the space $\mathbb{R}^3$
with Cartesian coordinates $(x,y,z)$, define vector fields
$$
X=\frac{\partial}{\partial x}, \, Y=\frac{\partial}{\partial y}+ x\frac{\partial}{\partial z}, \,
Z=\frac{\partial}{\partial z} \, .
$$
Then $\spann (X, Y, Z)$ is a Lie algebra (with the only nonzero bracket $[X,Y]=Z$), which is the
Lie algebra of $Nil$. Introduce a scalar product in such a way that the vectors $X, Y$ and $Z$ are
orthonormal. Consider the following unit vector field:
$$
\eta = \frac{xY+Z}{\sqrt{1+x^2}},
$$
and vector fields
$$
F_1=X, \, F_2=\frac{Y-xZ}{\sqrt{1+x^2}},
$$
which are orthogonal to $\eta$. In the notation of section \ref{ch2-1}, in each $p$ $F_1=X_1$,
$F_2=X_2-Z_2$, $\eta=X_3+Z_3$. By direct computation of covariant derivatives it can be shown that
the distribution spanned by $F_1$ and $F_2$ is integrable and form the tangent bundle of some
two-dimensional foliation $\mathcal{F}$ in $Nil$. From the computation of the second fundamental
form we obtain $\|B\|^2=\frac{\left(x^2-1\right)^2}{2\left(1+x^2\right)^2}$, and $H=0$. Thus the
leaves of this foliation are minimal surfaces. The Laplacian on $G=\eta$ is
$$
\Delta G =\left(
F_1F_1+F_2F_2-\left(\nabla_{F_1}F_1\right)^T-\left(\nabla_{F_2}F_2\right)^T
\right)G
$$
$$
=-\frac{x}{\left(1+x^2\right)^2} F_2-
\frac{1}{\left(1+x^2\right)^2} \eta.
$$
We obtain the same result considering \eqref{eq5} at some $p$. In
fact,
$$
2b_{22}\left|X_{3}\right|\left|Z_{3}\right|=0;
$$
$$
\left|X_{3}\right|^3\left|Z_{3}\right|-
2b_{21}\left|X_{3}\right|\left|Z_{3}\right|=\frac{x}{\left(1+x^2\right)^2};
$$
$$
\|B\|^2+\frac{1}{2}\left|Z_{3}\right|^2-
\frac{1}{2}\left|X_{3}\right|^2 +\left|X_{3}\right|^4-
2b_{21}\left|X_{3}\right|^2=\frac{1}{\left(1+x^2\right)^2}.
$$

In particular, foliation $\mathcal{F}$ gives an example of a CMC-surface in $Nil$ such that its
Gauss map is not harmonic.

\begin{proposition}\label{prop3}
Suppose that $M$ is a smooth oriented $2m$-dimensional manifold immersed in the $2m+1$-dimensional
Heisenberg group. If any two of the following three claims are true, then the third one is also
true.
\begin{enumerate}

\item\label{prop3-it1} $M$ is CMC;

\item\label{prop3-it2} the Gauss map of $M$ is harmonic;

\item\label{prop3-it3} at every point of $M$, the following holds:
\begin{equation}\label{eq5-1}
\left\{
\begin{array}{l}

\vphantom{\sum\limits_{i=1}^{n}} b_{2m \,k}=0,
1 \leqslant k \leqslant m-1, \, m+1 \leqslant k \leqslant 2m-1;\\

\vphantom{\sum\limits_{i=1}^{n}}
\left|Z_{2m+1}\right|\left(\left|X_{2m+1}\right|^2- 2b_{2m \,
m}\right)=0;\\

\vphantom{\sum\limits_{i=1}^{n}} \left|Z_{2m+1}\right|\left(b_{1
\, 1}+ \dots + b_{2m-1 \, 2m-1}
+ 3b_{2m \, 2m}\right)=0. \\

\end{array}
\right.
\end{equation}
Here $b_{ij}$, $1 \leqslant i,j \leqslant 2m$ are the coefficients
of the second fundamental form of $M$ in the basis chosen as
above.
\end{enumerate}
\end{proposition}

\begin{proof}
If \ref{prop3-it3} is true, then the equivalency of \ref{prop3-it1} and \ref{prop3-it2} immediately
follows from \ref{eq5}. Suppose \ref{prop3-it1} and \ref{prop3-it2} are true. Let $A$ be a set of
such points of $M$ that $\left|X_{2m+1}\right| \neq 0$. At the points of $A$ \ref{eq5} implies the
expressions in \eqref{eq5-1} . Since the distribution orthogonal to $Z$ is non-integrable, $A$ is
dense in $M$. Now the continuity of the left hand sides of the equations \eqref{eq5-1} implies
\ref{prop3-it3}.
\end{proof}

In the case $m=1$ the next theorem shows that the restrictions for
$M$ arising from \eqref{eq5-1} are rather strict.

\begin{theorem}\label{th2}
Let $M$ be a smooth oriented CMC-surface in the Heisenberg group $Nil$ whose Gauss map is harmonic.
Then $M$ is a ''cylinder'', that is, its position vector in the coordinates $x$, $y$, $z$ has the
form
\begin{equation}\label{eq6}
r(s,t)=(f_1(s),f_2(s),t),
\end{equation}
where $f_1$ and $f_2$ are some smooth functions.
\end{theorem}

\begin{proof}
For each $p \in M$ denote $a(p)=\left|X_3\right|$, $b(p)=\left|Z_3\right|$. Then $a$ and $b$ are
smooth scalar functions on $M$, and $a^2+b^2=1$. Consider an arbitrary point $p$ of $M$. Choose
$X_1$ as above and put $X_2=J(Z)X_1$. Denote by $T_1$ and $T_2$ the vector fields that at each $p
\in M$ are equal to $X_1$ and $X_2$ respectively. Consider unit tangent vector fields $F_1$ and
$F_2$, and a unit normal vector field $\eta$ of $M$ of the form
$$
F_1=T_1, \, F_2=bT_2-aZ, \, \eta=aT_2+bZ.
$$
Denote by $\kappa_1$ and $\kappa_2$ the geodesic curvatures of the integral curves of $F_1$ and
$F_2$ respectively. In other words,
\begin{equation}\label{eq6-pr1}
\overline{\nabla}_{F_1}F_1=\kappa_1F_2, \,
\overline{\nabla}_{F_1}F_2=-\kappa_1F_1,
\,\overline{\nabla}_{F_2}F_1=-\kappa_2F_2,
\,\overline{\nabla}_{F_2}F_2=\kappa_2F_1,
\end{equation}
where $\overline{\nabla}$ is the Riemannian connection on $M$
induced by the immersion. The Gaussian curvature of the surface is
\begin{equation}\label{eq6-pr2}
K=F_1(\kappa_2)+F_2(\kappa_1)-\left(\kappa_1\right)^2-
\left(\kappa_2\right)^2.
\end{equation}

Assume that for some $p \in M$ $a(p) \neq 0$ and $b(p) \neq 0$. Then $ab \neq 0$ on some
neighborhood $U$ of $p$. Then \eqref{eq5-1} implies that on $U$ the matrix of the second
fundamental form of $M$ is
\begin{equation}\label{eq6-pr3}
\left(
\begin{array}{cc}
  3H & \frac{1}{2}a^2 \\
  \frac{1}{2}a^2 & -H \\
\end{array}
\right) .
\end{equation}
In particular, the extrinsic curvature $K_{ext}$ of the surface is $-3H^2-\frac{1}{4}a^4$.

Denote by $B$ the second fundamental form of the immersion. Then
the Codazzi equations for $M$ are
$$
\left(\nabla_{F_1} B\right)(F_2,F_1)-\left(\nabla_{F_2}
B\right)(F_1,F_1)=\langle R(F_1,F_2)F_1,\eta\rangle=ab;
$$
$$
\left(\nabla_{F_2} B\right)(F_1,F_2)-\left(\nabla_{F_1}
B\right)(F_2,F_2)=\langle R(F_2,F_1)F_2,\eta\rangle=0.
$$
Computing the covariant derivatives of the second fundamental
form, we obtain for $U$
\begin{equation}\label{eq6-pr4}
\begin{array}{c}
aF_1(a)+4H\kappa_1 - a^2\kappa_2 - ab =0,\\
aF_2(a)-4H\kappa_2- a^2\kappa_1=0. \\
\end{array}
\end{equation}
The Gauss equation has the form
$$
K=K_{ext}+\langle R(F_1,F_2)F_2,F_1
\rangle=-3H^2-\frac{1}{4}a^4-\frac{3}{4}b^2+\frac{1}{4}a^2.
$$
From \eqref{eq6-pr2} we obtain
\begin{equation}\label{eq6-pr5}
F_1(\kappa_2)+F_2(\kappa_1)-\left(\kappa_1\right)^2-
\left(\kappa_2\right)^2=-3H^2-\frac{1}{4}a^4-\frac{3}{4}b^2+\frac{1}{4}a^2.
\end{equation}
Using \eqref{eq6-pr4} and the form of $F_2$ and $\eta$, we can
derive
$$
\langle \nabla_{F_1}F_2, \eta \rangle=-\langle F_2,
\nabla_{F_1}\eta \rangle=-\langle F_2, \nabla_{F_1}\left(
\frac{a}{b}F_2+\left( \frac{a^2}{b} + b \right)Z \right) \rangle
$$
$$
=-F_1\left(\frac{a}{b}\right)\langle F_2,F_2 \rangle - F_1\left(
\frac{1}{b}\right)\langle F_2,Z \rangle-\frac{a}{b}\langle F_2,
\nabla_{F_1}F_2 \rangle -\frac{1}{b} \langle F_2,\nabla_{F_1} Z
\rangle
$$
$$
=-F_1\left(\frac{a}{b}\right) + aF_1\left(
\frac{1}{b}\right)-\frac{1}{b} \langle F_2,-\frac{1}{2}T_2
\rangle=-\frac{1}{b}\left( -\frac{4H\kappa_1}{a} + a\kappa_2 + b
\right) +\frac{1}{2}
$$
$$
=-\frac{1}{2}+\frac{4H\kappa_1}{ab}-\frac{a}{b}\kappa_2;
$$
$$
\langle \nabla_{F_2}F_1, \eta \rangle=-\langle F_1,
\nabla_{F_2}\eta \rangle=-\langle F_1, \nabla_{F_2}\left(
\frac{a}{b}F_2+\left( \frac{a^2}{b} + b \right)Z \right) \rangle
$$
$$
=-\frac{a}{b}\langle F_1, \nabla_{F_2}F_2 \rangle -\frac{1}{b}
\langle F_1,\nabla_{F_2} Z \rangle=-\frac{a}{b} \kappa_2
-\frac{1}{b} \langle T_1,\frac{1}{2}bT_1 \rangle=-\frac{a}{b}
\kappa_2-\frac{1}{2}.
$$
In the above equations we used the fact that $Z$ is left invariant and the expressions \eqref{pr1}
for the covariant derivative. Since $ab \neq 0$, the integrability condition
$\langle[F_1,F_2],\eta\rangle=0$ takes the form $H\kappa_1=0$. Besides, \eqref{eq6-pr3} imply
$$
3H=b_{11}=\langle \nabla_{F_1}F_1, \eta \rangle=-\langle F_1,
\nabla_{F_1}\eta \rangle
$$
$$
=-\langle F_1, \nabla_{F_1}\left( \frac{a}{b}F_2+\left(
\frac{a^2}{b} + b \right)Z \right) \rangle=-\frac{a}{b}\langle
F_1, \nabla_{F_1}F_2 \rangle -\frac{1}{b} \langle F_1,\nabla_{F_1}
Z \rangle=\frac{a}{b} \kappa_1.
$$
Thus $H=\kappa_1=0$. In particular, $\nabla_{F_1}F_1=0$, hence $T_1=F_1$ is a geodesic vector field
in the ambient manifold. Note that $T_1$ belongs to the distribution that spans the left invariant
vector fields of $\mathcal{V}$. Considering the set of geodesics in $Nil$ (see \cite[proposition
(3.1), proposition (3.5)]{E1}), we obtain that $T_1=cX+dY$, where $c,d \in \mathbb{R}$ are some
constants, i.e., $T_1=X_1$ and $T_2=X_2$ are left invariant. Note that the second equation of
\eqref{eq6-pr4} implies $F_2(a)=F_2(b)=0$. Thus we obtain
$$
\nabla_{F_2}F_2=\nabla_{F_2}\left( bX_2-aZ
\right)=b\nabla_{bX_2-aZ}X_2-a\nabla_{bX_2-aZ}Z=-abX_1.
$$
Therefore $\kappa_2=-ab$. It follows from this equation, from the computations above in this proof,
and from \eqref{eq6-pr3} that
$$
\frac{1}{2}a^2=b_{12}=\langle \nabla_{F_1}F_2, \eta
\rangle=-\frac{a}{b}\kappa_2 -\frac{1}{2}=a^2-\frac{1}{2},
$$
and $a^2=b^2=\frac{1}{2}$. But then $a=b=\frac{\sqrt{2}}{2}$, and the first equation in
\eqref{eq6-pr4} implies $a\kappa_2+b=0$, which leads to a contradiction.

Thus $ab=0$ at each point of $M$. Since $a^2+b^2=1$ and $a$, $b$
are continuous, $a=1$ or $b=1$ identically. The latter case is
impossible because $Z^{\perp}$ is not integrable; then the normal
vector of $M$ is orthogonal to $\mathcal{Z}$, and $F_2=-Z$.
Therefore $M$ is invariant under the action of $\mathcal{Z}$ by
left translations, and $M$ is formed by integral curves of $Z$,
which are geodesics $(0,0,t)$. Then $M$ has the form \eqref{eq6}.
\end{proof}

Note that similar result for another definition of the Gauss map was obtained in \cite{Sa}. Also, in \cite{Sa} the
equations of the CMC-surfaces of the form \eqref{eq6} were obtained. Proposition \ref{prop3} then implies that the
Gauss maps of all these surfaces are harmonic.

\end{document}